\documentclass[a4paper]{amsart}
\usepackage{amsmath,amsthm,amssymb}

\PII{}
\makeatletter
\def\@setcopyright{\@empty}
\makeatother

\newcommand{\E}{E_n(f)_{p,\alpha,\beta}}
\newcommand{\Epar}[2]{E_{#1}\left(#2\right)_{p,\alpha,\beta}}
\newcommand{\T}[2]{\tilde T_{#1}\left(#2\right)}
\newcommand{\sT}{S_t(f,x,\nu,\mu)}
\newcommand{\w}{\tilde\omega(f,\delta)_{p,\alpha,\beta}}
\newcommand{\wpar}[1]{\tilde\omega\left(#1\right)_{p,\alpha,\beta}}
\newcommand{\norm}[1]{\left\|#1\right\|_{p,\alpha,\beta}}
\newcommand{\normpar}[2]{\left\|#1\right\|_{#2}}
\newcommand{\Si}[1]{\left(1-#1^2\right)}
\newcommand{\Px}[3]{P_{#1}^{(#2)}(#3)}
\newcommand{\Lp}{L_{p,\alpha,\beta}}
\newcommand{\Lmu}{L_{1,2,2}}
\newcommand{\allp}{1\le p\le\infty}

\newcommand{\Dx}{D_{x,\nu,\mu}}
\newcommand{\AD}{AD(p,\alpha,\beta)}
\newcommand{\sincost}{%
  \left(\sin\frac t2\right)^{2\nu+1}
  \left(\cos\frac t2\right)^{2\mu+1}}
\newcommand{\krn}[1]{%
  \left(\frac{\sin\frac{m#1}2}{\sin\frac{#1}2}\right)^{2q+4}}
\newcommand{\prn}[1]{\left(#1\right)}
\newcommand{\brc}[1]{\left\{#1\right\}}

\newtheorem{thm}{Theorem}
\newtheorem{lmm}{Lemma}
\newtheorem{cor}{Corollary}

\newcounter{const}

\numberwithin{const}{thm}
\numberwithin{const}{lmm}
\numberwithin{const}{cor}

\makeatletter
\newcommand{\Cn}[1][]{%
  \stepcounter{const}C_{\theconst}%
  \@ifnotempty{#1}{\newcounter{#1}\setcounter{#1}{\arabic{const}}}}
\makeatother
\newcommand{\refC}[1]{C_{\arabic{#1}}}
\newcommand{\lastC}{C_{\theconst}}
\newcommand{\prevC}[1][1]{%
	{\countdef\n=255
	 \n=\theconst
	 \advance\n by-#1
	 C_{\number\n}}}

\renewcommand{\theconst}{\arabic{const}}

\begin{document}

\title[Approximating by means of algebraic polynomials]
	{Approximating classes of functions defined by operators
	of differentiation or operators of generalised translation
	by means of algebraic polynomials}
\author{N.~Sh.\ Berisha}
\email{nimete-berisha@hotmail.com}
\author{F.~M.\ Berisha}
\address{F.~M.\ Berisha\\
	Faculty of Mathematics and Sciences\\
	University of Prishtina\\
	Mother Theresa av.~5\\
	10000 Prishtina\\
	Kosovo}
	\email{faton.berisha@uni-pr.edu}

\date{}

\begin{abstract}
	In this paper,
	approximation by means of algebraic polynomials
	of classes of functions
	defined by a generalised modulus of smoothness
	of operators of differentiation of these functions
	is considered.
	We give structural characteristics
	of classes of functions defined
	by the order of best approximation by algebraic polynomials.
\end{abstract}

\maketitle

\section{Introduction}

In a number of papers (see e.g.\
\cite{pawelke:acta-72,butzer-s-w:c-80,potapov:trudy-74,
  potapov:trudy-75,potapov:trudy-81,p-berisha:east-98,
potapov:mat-zametki-01})
approximation of classes of functions defined by symmetric
or asymmetric operators of generalised translation by
means of algebraic polynomials is considered.

In our paper we consider the approximation of classes of
functions defined by generalised modulus of smoothness
of derivatives of these functions.
In more general terms,
we consider approximation by algebraic polynomials
of certain generalised Lipschitz classes of functions.

By $L_p[a,b]$ we denote the set of functions~$f$
such that for $1\le p<\infty$
$f$~is a measurable function on the segment $[a,b]$ and
\begin{displaymath}
	\normpar f p
	=\biggl(\int_a^b|f(x)|^p\,dx\biggr)^{1/p}<\infty,
\end{displaymath}
and for $p=\infty$ the function~$f$ is continuos
on the segment $[a,b]$ and
\begin{displaymath}
	\normpar f\infty=\max_{a\le x\le b}|f(x)|.
\end{displaymath}
In case that $[a,b]=[-1,1]$ we simply write~$L_p$
instead of $L_p[-1,1]$.

Denote by~$\Lp$ the set of functions~$f$ such that
$f(x)\*(1-x)^\alpha\*(1+x)^\beta\in L_p$, and put
\begin{displaymath}
	\norm f=\|f(x)(1-x)^\alpha(1+x)^\beta\|_p.
\end{displaymath}

By $\E$ we denote the best approximation of
the function $f\in\Lp$ by means of
algebraic polynomials of degree not greater
than $n-1$, in $\Lp$ metrics, i.e.
\begin{displaymath}
	\E=\inf_{P_n}\norm{f-P_n},
\end{displaymath}
where~$P_n$ is an algebraic polynomial of degree
not greater than $n-1$.

For a summable function~$f$ we define the asymmetric
operator of generalised translation
$\T t{f,x}$ by
\begin{multline*}
	\T t{f,x}=\frac1{\pi\Si x}\\
		\times\int_{-1}^1\prn{1-R^2-2\Si z\sin^2t
			+4\Si{x}\Si{z}^2\sin^2t}f(R)\frac{dz}{\sqrt{1-z^2}},
\end{multline*}
where $R=x\cos t-z\sqrt{1-x^2}\sin t$.

For a function $f\in\Lp$ we define by means of this
operator of generalised translation the
generalised modulus of smoothness by
\begin{displaymath}
	\w=\sup_{|t|\le\delta}\norm{\T t{f,x}-f(x)}.
\end{displaymath}

We say that~$\varphi$ is a function of modulus of
continuity type if
\begin{enumerate}
\item $\varphi$ is continuos and non-negative function
  on the interval $(-1,1]$,
\item $\varphi(t_1)\le C_{\varphi,1}\varphi(t_2)$
  $(0<t_1\le t_2\le1)$,
\item $\varphi(2t)\le C_{\varphi,2}\varphi(t)$
  $(0<t\le\frac12)$.
\end{enumerate}

We say that a function~$f(x)$ has the derivative
of order~$r$ inside of the interval $(-1,1)$
if the function~$f(x)$ has the absolutely continuos
derivative of order $r-1$ in every segment
$[a,b]\subset(-1,1)$. From the last condition
it follows that almost everywhere on the segment $[a,b]$
there exists the finite derivative of order~$r$, which
is a summable function on that interval.

Denote by $\Dx$ the following operator of differentiation
\begin{displaymath}
\Dx=\Si x\frac{d^2}{dx^2}
  +(\mu-\nu-(\nu+\mu+2)x)\frac d{dx},
\end{displaymath}
and put
\begin{align*}
\Dx^1f(x) &=\Dx f(x),\\
\Dx^r f(x)&=\Dx\prn{\Dx^{r-1}f(x)} \quad (r=1,2,\dotsc).
\end{align*}

We say that $f\in\AD$ if $f\in\Lp$, the function~$f$
has the derivative $\frac d{dx}f(x)$ absolutely continuos
on every segment $[a,b]\subset(-1,1)$ and $\Dx f(x)\in\Lp$.

By $\Px n{\nu,\mu}x$ $(n=0,1,\dotsc)$ we denote the
Jacobi's polynomials, i.e.\ algebraic polynomials of
order~$n$, orthogonal to each other with weight
$(1-x)^{\nu}(1+x)^{\mu}$ on the segment $[-1,1]$
and normed by the condition $\Px n{\nu,\mu}1=1$
$(n=0,1,\dotsc)$.

Let $\nu\ge\mu\ge-\frac12$. The following symmetric
operators of generalised translation
(see e.g.\ \cite{pawelke:acta-72,butzer-s-w:c-80,
  potapov:trudy-74, potapov:trudy-75,potapov:trudy-81}))
will have an auxiliary role below:
\begin{enumerate}
\item for $\nu=\mu=-\frac12$
\begin{displaymath}
\sT=\frac12(f(Q_{x,t,1,1})-f(Q_{x,-t,1,1}));
\end{displaymath}
\item for $\nu=\mu>-\frac12$
\begin{displaymath}
\sT=\frac1{\gamma(\nu)}
  \int_{-1}^1f(Q_{x,t,z,1})\Si{z}^{\nu-\frac12}\,dz;
\end{displaymath}
\item for $\nu>\mu=-\frac12$
\begin{displaymath}
\sT=\frac1{\gamma(\nu)}
  \int_{-1}^1f(Q_{x,t,1,z})\Si{z}^{\nu-\frac12}\,dz;
\end{displaymath}
\item for $\nu>\mu>-\frac12$
\begin{multline*}
\sT\\
=\frac1{\gamma(\nu,\mu)}
  \int_0^1\int_{-1}^1f(Q_{x,t,z,u})
	  \Si{z}^{\nu-\mu-1}z^{2\mu+1}
	  \Si{u}^{\mu-\frac12}\,du\,dz,
\end{multline*}
\end{enumerate}
where
\begin{align*}
Q_{x,t,z,u}
  & =x\cos t+zu\sqrt{1-x^2}\sin t
    -\Si u(1-x)\sin^2\frac t2,\\
\gamma(\nu)
  & =\int_{-1}^1\Si{z}^{\nu-\frac12}\,dz,\\
\gamma(\nu,\mu)
  & =\int_0^1\int_{-1}^1\Si{z}^{\nu-\mu-1}
	    z^{2\mu+1}\Si{u}^{\mu-\frac12}\,du\,dz.
\end{align*}

\section{Auxiliary statements}

We need the following lemmas in order to prove our results.

\begin{lmm}\label{lm:bernshtein-markov}
Let $P_n(x)$ be an algebraic polynomial of order
not greater than $n-1$, $\allp$, $\rho\ge0$, $\sigma\ge0$;
\begin{alignat*}3
\alpha &>-\frac1p, &\quad \beta &>-\frac1p
  &\quad &\text{for $1\le p<\infty$},\\
\alpha &\ge0,      &\quad \beta &\ge0
  &\quad &\text{for $p=\infty$}.
\end{alignat*}
The following inequalities hold true
\begin{gather*}
\normpar{P'_n(x)}{p,\alpha+\frac12,\beta+\frac12}
  \le\Cn n\norm{P_n},\\
\norm{P_n}
  \le\Cn n^{2\max\{\rho,\sigma\}}
    \normpar{P_n}{p,\alpha+\rho,\beta+\sigma},
\end{gather*}
where constants~$\prevC$ and~$\lastC$ do not depend on~$n$.
\end{lmm}

Lemma~\ref{lm:bernshtein-markov} is proved
in~\cite{halilova:izv-74}.

\begin{lmm}\label{th:bound-T}
Let be given numbers~$p$, $\alpha$, $\beta$ and~$\gamma$
such that $\allp$, $\gamma=\min\{\alpha,\beta\}$;
\begin{alignat*}2
\gamma &>1-\frac1{2p} &\quad &\text{for $1\le p<\infty$},\\
\gamma &\ge1          &\quad &\text{for $p=\infty$}.
\end{alignat*}
Let~$\varepsilon$ be an arbitrary number from the interval
$0<\varepsilon<\frac12$ and let
\begin{displaymath}
\gamma_1=
\begin{cases}
\alpha-\beta & \text{if $\alpha>\beta$}\\
0            & \text{if $\alpha\le\beta$},
\end{cases}
\quad \gamma_2=
\begin{cases}
0            & \text{if $\alpha>\beta$}\\
\beta-\alpha & \text{if $\alpha\le\beta$};
\end{cases}
\end{displaymath}
for $1<p\le\infty$
\begin{displaymath}
\gamma_3=
\begin{cases}
\gamma-\frac32+\frac1{2p}+\varepsilon
  & \text{if $\gamma\ge\frac32-\frac1{2p}$}\\
0
  & \text{if $\gamma<\frac32-\frac1{2p}$},
\end{cases}
\end{displaymath}
for $p=1$
\begin{displaymath}
\gamma_3=
\begin{cases}
\gamma-1 & \text{if $\gamma\ge1$}\\
0        & \text{if $\gamma<1$}.
\end{cases}
\end{displaymath}
Then the following inequality holds true
\begin{multline*}
\norm{\T t{f,x}}
  \le C\Big(
      \norm f+t^{2(\gamma_1+\gamma_2)}
      \normpar f{p,\alpha-\gamma_1,\beta-\gamma_2}\\
+t^{2\gamma_3}\normpar f{p,\alpha-\gamma_3,\beta-\gamma_3}
	  +t^{2(\gamma_1+\gamma_2+\gamma_3)}
		  \normpar f{p,\alpha-\gamma_1-\gamma_3,\beta-\gamma_2-\gamma_3}
	\Big),
\end{multline*}
where constant~$C$ does not depend on~$f$ and~$t$.
\end{lmm}

Lemma~\ref{th:bound-T} is proved in~\cite{p-berisha:east-98}.

\begin{lmm}\label{th:T-Q}
Let be given positive integers~$q$ and~$m$ and let $f\in\Lmu$.
The function
\begin{displaymath}
Q(x)=\frac1{\gamma_m}\int_0^\pi\T t{f,x}\krn t\sin^3t\,dt,
\end{displaymath}
where
\begin{displaymath}
\gamma_m=\int_0^\pi\krn t\sin^3t\,dt,
\end{displaymath}
is an algebraic polynomial of degree not greater than
$(q+2)\*(m-1)$.
\end{lmm}

Lemma~\ref{th:T-Q} is also proved in~\cite{p-berisha:east-98}

\begin{lmm}\label{lm:rho-sigma}
Let $f\in\Lp$ and let be given numbers~$p$, $\alpha$, $\beta$,
$\rho$ and~$\sigma$ such that $\allp$, $\rho\ge0$, $\sigma\ge0$;
\begin{alignat*}3
\alpha &>-\frac1p, &\quad \beta &>-\frac1p
  &\quad &\text{for $1\le p<\infty$},\\
\alpha &\ge0,      &\quad \beta &\ge0
  &\quad &\text{for $p=\infty$}.
\end{alignat*}
Let~$\varphi$ be a function of modulus of continuity type
such that
\begin{equation}\label{eq:phi-3}
\sum_{j=n+1}^\infty j^{2\lambda_0-1}\varphi\prn{\frac1j}
\le C_{\varphi,3}n^{2\lambda_0}\varphi\prn{\frac1n},
\end{equation}
where $\lambda_0=\max\{\rho,\sigma\}$ and
constant~$C_{\varphi,3}$ does not depend on~$n$.
If there exists a sequence of algebraic polynomials $P_n(x)$
of degree not greater than $n-1$ $(n=0,1,\dotsc)$
such that
\begin{displaymath}
\normpar{f-P_n}{p,\alpha+\rho,\beta+\sigma}
\le\Cn\varphi\prn{\frac1n},
\end{displaymath}
then there exists a sequence of algebraic polynomials $R_n(x)$
of degree not greater than $n-1$ $(n=0,1,\dotsc)$ such that
\begin{displaymath}
\norm{f-R_n}\le\Cn n^{2\lambda_0}\varphi\prn{\frac1n},
\end{displaymath}
where constants~$\prevC$ and~$\lastC$ do not depend on~$f$
and~$n$. Also we have
\begin{displaymath}
R_{2^N}(x)=P_{2^N}(x).
\end{displaymath}
\end{lmm}

\begin{proof}
We consider the sequence of algebraic polynomials $Q_n(x)$
of degree not greater than $2^n-1$ given by
\begin{displaymath}
Q_k(x)=P_{2^k}(x)-P_{2^{k-1}}(x) \quad (k=1,2,\dotsc)
\end{displaymath}
and $Q_0(x)=P_1(x)$. From the conditions of the lemma
it follows that
\begin{multline*}
\normpar{Q_k}{p,\alpha+\rho,\beta+\sigma}
\le\normpar{P_{2^k}-f}{p,\alpha+\rho,\beta+\sigma}
	+\normpar{f-P_{2^{k-1}}}{p,\alpha+\rho,\beta+\sigma}\\
\le\Cn\prn{\varphi\prn{\frac1{2^k}}
    +\varphi\prn{\frac1{2^{k-1}}}}.
\end{multline*}
Considering the properties of the function~$\varphi$
we get
\begin{displaymath}
\normpar{Q_k}{p,\alpha+\rho,\beta+\sigma}
\le\Cn\varphi\prn{\frac1{2^k}}.
\end{displaymath}
Applying Lemma~\ref{lm:bernshtein-markov} and that
evaluate we obtain
\begin{displaymath}
\norm{Q_k}
\le\Cn[cn:normQk]2^{2k\lambda_0}
  \varphi\prn{\frac1{2^k}}.
\end{displaymath}
There from
\begin{displaymath}
\sum_{k=0}^\infty\norm{Q_k}
\le\lastC\sum_{k=0}^\infty 2^{2k\lambda_0}
  \varphi\prn{\frac1{2^k}}.
\end{displaymath}
Note that considering the properties of the
function~$\varphi$ we have
\begin{multline*}
\sum_{j=2^k}^{2^{k+1}-1}j^{2\lambda_0-1}
      \varphi\prn{\frac1j}
  \ge C_{\varphi,1}^{-1}C_{\varphi,2}^{-1}
    \varphi\prn{\frac1{2^k}}
	  \sum_{j=2^k}^{2^{k+1}-1}j^{2\lambda_0-1}\\
\ge\Cn\varphi\prn{\frac1{2^k}}2^k2^{k(\lambda_0-1)}
  =\lastC2^{2k\lambda_0}\varphi\prn{\frac1{2^k}}.
\end{multline*}
So, we get
\begin{displaymath}
\sum_{k=0}^\infty\norm{Q_k}
\le\Cn\sum_{k=0}^\infty\sum_{j=2^k}^{2^{k+1}-1}
    j^{2\lambda_0-1}\varphi\prn{\frac1j}
=\lastC\sum_{k=0}^\infty k^{2\lambda_0-1}
    \varphi\prn{\frac1k}.
\end{displaymath}
Thus, inequality~\eqref{eq:phi-3} yields
\begin{displaymath}
\sum_{k=0}^\infty\norm{Q_k}<\infty.
\end{displaymath}
Hence, considering the conditions of the lemma
it follows that the series $\sum_{k=0}^\infty Q_k(x)$
converge to $f(x)$ in terms of $L_p[a,b]$ for every
segment $[a,b]\subset(-1,1)$.

Now we consider the expression
\begin{displaymath}
I=\norm{f-P_{2^N}}.
\end{displaymath}
From what we said above it follows that
\begin{multline*}
I\le\sum_{k=N+1}^\infty\norm{Q_k}
  \le\refC{cn:normQk}\sum_{k=N+1}^\infty
    2^{2k\lambda_0}\varphi\prn{\frac1{2^k}}\\
\le\Cn\sum_{k=N+1}^\infty\sum_{j=2^k}^{2^{k+1}-1}
    j^{2\lambda_0-1}\varphi\prn{\frac1j}
  =\lastC\sum_{k=2^{N+1}}^\infty
    k^{2\lambda_0-1}\varphi\prn{\frac1k}.
\end{multline*}
Considering the inequality~\eqref{eq:phi-3}
and the properties of the function~$\varphi$ we
obtain that
\begin{displaymath}
I\le\Cn2^{2(N+1)\lambda_0}\varphi\prn{\frac1{2^{N+1}}}
\le\Cn2^{2N\lambda_0}\varphi\prn{\frac1{2^N}},
\end{displaymath}
where constant~$\lastC$ does not depend on~$f$ and~$N$.

Put
\begin{displaymath}
R_n(x)=P_{2^N}(x) \quad (2^{N-1}<n\le2^N),
\end{displaymath}
we get
\begin{displaymath}
\norm{f-R_n}
\le\lastC2^{2N\lambda_0}\varphi\prn{\frac1{2^N}}
\le\Cn n^{2\lambda_0}\varphi\prn{\frac1n}.
\end{displaymath}

Lemma~\ref{lm:rho-sigma} is proved.
\end{proof}

\begin{lmm}\label{lm:E-D}
Let be given numbers~$p$, $\alpha$, $\beta$, $\nu$,
and~$\mu$ such that $\allp$, $\nu\ge\mu\ge-\frac12$;
\begin{enumerate}
\item if $\nu=\mu=-\frac12$, then $\alpha=\beta=-\frac1{2p}$;
\item if $\nu=\mu>-\frac12$, then $\alpha=\beta$, and
\begin{alignat*}2
-\frac12    &<\alpha\le\nu
  &\quad &\text{for $p=1$},\\
-\frac1{2p} &<\alpha<\nu+\frac12-\frac1{2p}
  &\quad &\text{for $1<p<\infty$},\\
0           &\le\alpha<\nu+\frac12
  &\quad &\text{for $p=\infty$};
\end{alignat*}
\item if $\nu>\mu=-\frac12$, then $\beta=-\frac1{2p}$, and
\begin{alignat*}2
-\frac12    &<\alpha\le\nu
  &\quad &\text{for $p=1$},\\
-\frac1{2p} &<\alpha<\nu+\frac12-\frac1{2p}
  &\quad &\text{for $1<p<\infty$},\\
0           &\le\alpha<\nu+\frac12
  &\quad &\text{for $p=\infty$};
\end{alignat*}
\item if $\nu>\mu>-\frac12$, then $\nu-\mu>\alpha-\beta\ge0$,
and
\begin{alignat*}2
-\frac12    &<\beta\le\mu
  &\quad &\text{for $p=1$},\\
-\frac1{2p} &<\beta<\mu+\frac12-\frac1{2p}
  &\quad &\text{for $1<p<\infty$},\\
0           &\le\beta<\mu+\frac12
  &\quad &\text{for $p=\infty$}.
\end{alignat*}
\end{enumerate}
For $f(x)\in\AD$ the following inequality holds true
\begin{displaymath}
\E\le \frac{C}{n^2}\norm{\Dx f(x)},
\end{displaymath}
where constant~$C$ does not depend on~$f$ and~$n$.
\end{lmm}

\begin{proof}
We choose the positive integer~$q$ such that $q>\nu$.
For every positive integer~$n$
we choose the positive integer~$m$ such that
\begin{displaymath}
\frac{n-1}{q+2}<m\le\frac{n-1}{q+2}+1.
\end{displaymath}
In~\cite{potapov:trudy-74} and~\cite{potapov:trudy-75}
it is proved that the function
\begin{displaymath}
Q(x)=\frac1{\gamma_m}\int_0^\pi\sT\krn t\sincost\,dt,
\end{displaymath}
where
\begin{displaymath}
\gamma_m=\int_0^\pi\krn t\sincost\,dt,
\end{displaymath}
is an algebraic polynomial of degree not greater than $n-1$.
Applying the generalised Minkowski's inequality we get
\begin{multline*}
\E\le\norm{f-Q}
\le\frac1{\gamma_m}\int_0^\pi\norm{\sT-f(x)}\\
	\times\krn t\sincost\,dt.
\end{multline*}
In~\cite[p.~47]{potapov:vestnik-83} it is proved that
under the conditions of the lemma we have
\begin{displaymath}
\norm{\sT-f(x)}\le\Cn t^2\norm{\Dx f(x)},
\end{displaymath}
where constant~$\lastC$ does not depend on~$f$ and~$t$.
Hence we get
\begin{multline*}
\E\le\lastC\norm{\Dx f(x)}\\
\times\frac1{\gamma_m}\int_0^\pi t^2\krn t\sincost\,dt.
\end{multline*}
Applying a standard estimate of Jackson's
kernel \cite[p.~233--235]{nikolskii:priblizhenie}
we obtain
\begin{displaymath}
\E\le\frac{\Cn}{m^2}\norm{\Dx f(x)}
\le\frac{\Cn}{n^2}\norm{\Dx f(x)}.
\end{displaymath}

Lemma~\ref{lm:E-D} is proved.
\end{proof}

\begin{cor}\label{cr:E-D}
Let numbers~$p$, $\alpha$, $\beta$, $\nu$, and~$\mu$
satisfy the conditions of Lemma~\ref{lm:E-D}.
For $f(x)\in\AD$ the following inequality holds true
\begin{displaymath}
\E\le\frac C{n^2}\Epar n{\Dx f},
\end{displaymath}
where constant~$\lastC$ does not depend on~$f$ and~$n$.
\end{cor}

\begin{proof}
Let $P_n(x)$ be the algebraic polynomial of best approximation
of the function $\Dx f(x)$ of degree
not greater than $n-1$. It is obvious that the
polynomial $P_n(x)$ may be written in the following form
\begin{displaymath}
P_n(x)=\sum_{k=0}^{n-1}\lambda_k\Px k{\nu,\mu}x.
\end{displaymath}
Put
\begin{displaymath}
g(x)=f(x)
+\sum_{k=0}^{n-1}\frac{\lambda_k}{k(k+\nu+\mu+1)}
  \Px k{\nu,\mu}x.
\end{displaymath}
From Lemma~\ref{lm:E-D} it follows
that \cite[p.~171]{erdelyi-m-o-t:transcendental}
\begin{multline*}
\Epar n g\le\frac{\Cn}{n^2}\norm{\Dx g(x)}\\
=\frac{\lastC}{n^2}
    \norm{\Dx f(x)
      +\sum_{k=0}^{n-1}\frac{\lambda_k}{k(k+\nu+\mu+1)}
        \Dx P_k^{(\nu,\mu)}(x)}\\
=\frac{\lastC}{n^2}
    \norm{\Dx f(x)
      -\sum_{k=0}^{n-1}\lambda_k P_k^{(\nu,\mu)}(x)}
  =\frac{\lastC}{n^2}\Epar n{\Dx f}.
\end{multline*}
Thus, considering that the function $f(x)-g(x)$ is an
algebraic polynomial of degree not greater than $n-1$,
we obtain
\begin{multline*}
\E\le\Epar n{f-g}+\Epar n{g}=\Epar n{g}\\
\le\frac{\lastC}{n^2}\Epar n{\Dx f}.
\end{multline*}

The corollary is proved.
\end{proof}

Note that an analogue to the corollary
is given in~\cite{potapov:steklov-01}.

\section{Statements of results}

Now we formulate and prove our results.

\begin{thm}\label{th:E-ED}
Let be given numbers~$p$, $\alpha$, $\beta$, $\nu$, $\mu$
and~$r$ such that $\allp$, $r\in\mathbb N$,
$\nu\ge\mu\ge-\frac12$;
\begin{enumerate}
\item if $\nu=\mu=-\frac12$, then $\alpha=\beta=-\frac1{2p}$;
\item if $\nu=\mu>-\frac12$, then $\alpha=\beta$,
and
\begin{alignat*}2
-\frac12    &<\alpha\le\nu
  &\quad &\text{for $p=1$},\\
-\frac1{2p} &<\alpha<\nu+\frac12-\frac1{2p}
  &\quad &\text{for $1<p<\infty$},\\
0           &\le\alpha<\nu+\frac12
  &\quad &\text{for $p=\infty$};
\end{alignat*}
\item if $\nu>\mu=-\frac12$, then $\beta=-\frac1{2p}$,
and
\begin{alignat*}2
-\frac12    &<\alpha\le\nu
  &\quad &\text{for $p=1$},\\
-\frac1{2p} &<\alpha<\nu+\frac12-\frac1{2p}
  &\quad &\text{for $1<p<\infty$},\\
0           &\le\alpha<\nu+\frac12
  &\quad &\text{for $p=\infty$};
\end{alignat*}
\item if $\nu>\mu>-\frac12$, then $\nu-\mu>\alpha-\beta\ge0$,
and
\begin{alignat*}2
-\frac12    &<\beta\le\mu
  &\quad &\text{for $p=1$},\\
-\frac1{2p} &<\beta<\mu+\frac12-\frac1{2p}
  &\quad &\text{for $1<p<\infty$},\\
0 				&\le\beta<\mu+\frac12
  &\quad &\text{for $p=\infty$}.
\end{alignat*}
\end{enumerate}
Let~$\varphi$ be a function of modulus of continuity
type such that
\begin{displaymath}
\sum_{j=n+1}^\infty\frac1j\varphi\prn{\frac1j}
\le\Cn\varphi\prn{\frac1n},
\end{displaymath}
where constant~$\lastC$ does not depend on~$n$.
Let $f(x)\in\Lp$. Necessary and sufficient condition
for the function~$f(x)$ to have the derivative of
order $2r-1$ inside of the interval $(-1,1)$ and
\begin{displaymath}
\Epar n{\Dx^r f}\le\Cn\varphi\prn{\frac1n}
\end{displaymath}
is that the following inequality is satisfied
\begin{displaymath}
\E\le\Cn n^{-2r}\varphi\prn{\frac1n},
\end{displaymath}
where constants~$\prevC$ and~$\lastC$ do not depend on~$f$
and~$n$.
\end{thm}

\begin{proof}
The necessity of the condition is implied by induction
directly from Corollary~\ref{cr:E-D}.
We prove that the condition is sufficient.

Let $P_n(x)$ be the algebraic polynomial of
best approximation of the function~$f$.
We consider the sequence of polynomials $Q_k(x)$ given by
\begin{displaymath}
Q_k(x)=P_{2^k}(x)-P_{2^{k-1}}(x) \quad (k=1,2,\dotsc)
\end{displaymath}
and $Q_0(x)=P_1(x)$. From the conditions of the theorem,
considering the properties of the function~$\varphi$
for $k\ge1$ it follows that
\begin{multline}\label{eq:Qk-r}
\norm{Q_k}=\norm{P_{2^k}-P_{2^{k-1}}}
  \le\Epar{2^k}f+\Epar{2^{k-1}}f\\
\le2\Epar{2^{k-1}}f
  \le\Cn 2^{-2(k-1)r}\varphi\prn{\frac1{2^{k-1}}}
  \le\Cn 2^{-2kr}\varphi\prn{\frac1{2^k}}.
\end{multline}
Applying Lemma~\ref{lm:bernshtein-markov} twice we get
\begin{multline*}
\norm{\Dx Q_k(x)}
  \le\normpar{Q''_k(x)}{p,\alpha+1,\beta+1}
    +(|\mu-\nu|+|\nu+\mu+2|)\norm{Q'_k(x)}\\
\le\Cn 2^{2k}\norm{Q_k},
\end{multline*}
where constant~$\lastC$ does not depend on~$k$.
Applying this inequality $r$~times we obtain
\begin{displaymath}
\norm{\Dx^r Q_k(x)}
\le\Cn[cn:DQk]2^{2kr}\norm{Q_k}.
\end{displaymath}
Thus inequality~\eqref{eq:Qk-r} yields
\begin{displaymath}
\sum_{k=1}^N\norm{\Dx^r Q_k(x)}
\le\Cn\sum_{k=1}^N\varphi\prn{\frac1{2^k}}.
\end{displaymath}
Noting that
\begin{displaymath}
\sum_{j=2^k}^{2^{k+1}-1}\frac1j\varphi\prn{\frac1j}
\ge C_{\varphi,1}^{-1}C_{\varphi,2}^{-1}
  \varphi\prn{\frac1{2^k}}\sum_{j=2^k}^{2^{k+1}-1}\frac1j
\ge\Cn\varphi\prn{\frac1{2^k}},
\end{displaymath}
considering the conditions of the theorem we have
\begin{displaymath}
\sum_{k=1}^\infty\norm{\Dx^r Q_k(x)}
\le\Cn\sum_{k=1}^\infty\sum_{j=2^k}^{2^{k+1}-1}
  \frac1j\varphi\prn{\frac1j}
\le\lastC\sum_{k=1}^\infty\frac1k\varphi\prn{\frac1k}
<\infty.
\end{displaymath}

Since
\begin{displaymath}
\sum_{k=0}^n Q_k(x)=P_{2^n}(x),
\end{displaymath}
from the inequality~\eqref{eq:Qk-r} and the conditions
of the theorem it follows that for every
segment $[a,b]\subset(-1,1)$ the series
$\sum_{k=0}^\infty Q_k(x)$ converges in terms of
$L_p[a,b]$ metrics to the function~$f(x)$.
Since the series
\begin{displaymath}
\sum_{k=0}^\infty\Dx^r Q_k(x)
\end{displaymath}
also converges in terms of $L_p[a,b]$ metrics,
then \cite[p.~202]{nikolskii:priblizhenie}
these series converge to the function $\Dx^r f(x)$.
This way we showed that the function~$f(x)$
has the derivative of order $2r-1$ absolutely
continuos on every segment $[a,b]\subset(-1,1)$.

Now we estimate the expression
\begin{displaymath}
I=\norm{\Dx^r f(x)-\Dx^r P_{2^N}(x)}.
\end{displaymath}
From what we said above it is obvious that
\begin{multline*}
I\le\sum_{k=N+1}^\infty\norm{\Dx^r Q_k(x)}
  \le\refC{cn:DQk}
    \sum_{k=N+1}^\infty 2^{2kr}\norm{Q_k}\\
\le\Cn\varphi\prn{\frac1{2^k}}
  \le\Cn\sum_{k=2^{N+1}}^\infty
      \frac1k\varphi\prn{\frac1k}.
\end{multline*}
Hence we conclude that
\begin{displaymath}
I\le\Cn\varphi\prn{\frac1{2^{N+1}}}
\le\Cn\varphi\prn{\frac1{2^N}}.
\end{displaymath}

Put
\begin{displaymath}
R_n(x)=\Dx^r P_{2^N}(x) \quad (2^N\le n<2^{N+1});
\end{displaymath}
we have
\begin{multline*}
\Epar n{\Dx^r f}\le\norm{\Dx^r f(x)-R_n(x)}\\
\le\lastC\varphi\prn{\frac1{2^N}}
  \le\Cn\varphi\prn{\frac1n}.
\end{multline*}

Theorem~\ref{th:E-ED} is proved.
\end{proof}

Note that for a power function $\varphi(\delta)=\delta^\lambda$,
the assertion of the theorem is given in~\cite{potapov:dokl-05}.

\begin{thm}\label{th:HsubE}
Let be given a function~$\varphi$ of modulus
of continuity type and numbers~$p$,
$\alpha$ and~$\beta$ such that $\allp$;
\begin{alignat*}3
\alpha &\le2,       &\quad \beta &\le2
  &\quad &\text{for $p=1$},\\
\alpha &<3-\frac1p, &\quad \beta &<3-\frac1p
  &\quad &\text{for $1<p\le\infty$}.
\end{alignat*}
Let $f\in\Lp$. If
\begin{displaymath}
\w\le M\varphi(\delta),
\end{displaymath}
then
\begin{displaymath}
\E\le CM\varphi\prn{\frac1n},
\end{displaymath}
where constant~$C$ does not depend on~$f$, $M$ dhe~$n$.
\end{thm}

\begin{proof}
From the properties of the function~$\varphi$
it follows that there exists a constant~$\gamma$
such that for every $l>0$ the following inequality
is satisfied
\begin{displaymath}
\varphi(lt)\le\Cn[cn:gamma](l+1)^\gamma\varphi(t),
\end{displaymath}
where constant~$\lastC$ does not depend on~$l$ and~$t$.

Indeed, if $l<1$, then
\begin{displaymath}
\varphi(lt)\le C_{\varphi,1}\varphi(t),
\end{displaymath}
i.e.\ we get $\gamma\ge0$.
If $l\ge1$, then choosing the positive integer~$m$
such that
\begin{displaymath}
2^{m-1}\le l<2^m
\end{displaymath}
we have
\begin{displaymath}
\varphi(lt)\le C_{\varphi,1}\varphi(2^m t)
\le C_{\varphi,1}C_{\varphi,2}^m\varphi(t).
\end{displaymath}
We choose the positive integer~$N$ such that
\begin{displaymath}
2^{N-1}\le C_{\varphi,2}<2^N,
\end{displaymath}
getting
\begin{displaymath}
\varphi(lt)\le C_{\varphi,1}2^{Nm}\varphi(t)
=C_{\varphi,1}2^N2^{N(m-1)}\varphi(t)
\le\Cn(l+1)^N\varphi(t),
\end{displaymath}
i.e.\ $\gamma\ge N$.

We choose a $\gamma>0$ and a positive integer~$q$
such that $2q>\gamma$, and for every positive
integer~$n$ we choose the positive integer~$m$
satisfying the condition
\begin{equation}\label{eq:m}
\frac{n-1}{q+2}<m\le\frac{n-1}{q+2}+1.
\end{equation}
It is easy to prove that under the condition of
the theorem we have $f\in\Lmu$.
Thus, for those~$q$ and~$m$ the algebraic polynomial~$Q(x)$
defined in Lemma~\ref{th:T-Q}
is an algebraic polynomial of degree not greater than $n-1$.
Hence
\begin{multline*}
\E\le\norm{f(x)-Q(x)}\\
=\norm{\frac1{\gamma_m}
  \int_0^\pi\prn{f(x)-\T t{f,x}}\krn t\sin^3t\,dt}.
\end{multline*}
Applying the generalised Minkowski's inequality we obtain
\begin{displaymath}
\E\le\frac1{\gamma_m}
  \int_0^\pi\norm{\T t{f,x}-f(x)}\krn t\sin^3t\,dt.
\end{displaymath}
There from by the conditions of the theorem we get
\begin{displaymath}
\E\le\frac M{\gamma_m}
  \int_0^\pi\varphi(t)\krn t\sin^3t\,dt.
\end{displaymath}
Since
\begin{displaymath}
\varphi(t)=\varphi\prn{nt\cdot\frac1n}
\le\refC{cn:gamma}(1+nt)^\gamma\varphi\prn{\frac1n},
\end{displaymath}
we have
\begin{multline*}
\E\le\refC{cn:gamma}\frac M{\gamma_m}\varphi\prn{\frac1n}
	\int_0^\pi(1+nt)^\gamma\krn t\sin^3t\,dt\\
\le\Cn M\varphi\prn{\frac1n}
	\brc{1+\frac{n^\gamma}{\gamma_m}
	  \int_0^\pi t^\gamma\krn t\sin^3t\,dt}.
\end{multline*}
Applying now the standard estimate of Jackson's kernel
and the inequality~\eqref{eq:m} we obtain
\begin{displaymath}
\E\le\Cn M\varphi\prn{\frac1n}(1+n^\gamma m^{-\gamma})
\le\Cn M\varphi\prn{\frac1n}.
\end{displaymath}

Theorem~\ref{th:HsubE} is proved.
\end{proof}

\begin{thm}\label{th:EsubH}
Let be given numbers~$p$, $\alpha$ and~$\beta$ such that
$\allp$;
\begin{alignat*}3
\alpha &>1-\frac1{2p}, &\quad \beta &>1-\frac1{2p}
  &\quad &\text{for $1\le p<\infty$},\\
\alpha &\ge1,          &\quad \beta &\ge1
  &\quad &\text{for $p=\infty$}.
\end{alignat*}
Let~$\varphi$ be a function of modulus of continuity
type such that inequality~\eqref{eq:phi-3} for
\begin{displaymath}
\lambda_0
=\max\brc{
  |\alpha-\beta|,\alpha-\frac32+\frac1{2p},
  \beta-\frac32+\frac1{2p}},
\end{displaymath}
and inequality
\begin{equation}\label{eq:phi-4}
\sum_{j=1}^n j\varphi\prn{\frac1j}
\le C_{\varphi,4}n^2\varphi\prn{\frac1n}
\end{equation}
are satisfied, where constant~$C_{\varphi,4}$ does not
depend on~$n$. Let $f\in\Lp$.
If
\begin{displaymath}
\E\le M\varphi\prn{\frac1n},
\end{displaymath}
then
\begin{displaymath}
\w\le CM\varphi(\delta),
\end{displaymath}
where constant~$C$ does not depend on~$f$, $M$ and~$\delta$.
\end{thm}

\begin{proof}
Let $P_n(x)$ be the algebraic polynomial of
best approximation of degree not greater
than $n-1$ of the function~$f$.
Let the polynomials $Q_k(x)$ be given by
\begin{displaymath}
Q_k(x)=P_{2^k}(x)-P_{2^{k-1}}(x) \quad (k=1,2,\dotsc)
\end{displaymath}
and $Q_0(x)=P_1(x)$. Since for $k\ge1$ we have
\begin{displaymath}
\norm{Q_k}\le\Epar{2^k}f+\Epar{2^{k-1}}f,
\end{displaymath}
considering the conditions of the theorem we have
\begin{equation}\label{eq:Qk}
\norm{Q_k}\le\Cn M\varphi\prn{\frac1{2^k}}.
\end{equation}

We estimate the expression
\begin{displaymath}
I=\norm{\T t{f,x}-f(x)}.
\end{displaymath}
Let $0<|t|\le\delta$. Since the operator $\T t{f,x}$ is linear,
for every positive integer~$N$ we have
\begin{displaymath}
I\le\norm{\T t{f-P_{2^N},x}-\prn{f(x)-P_{2^N}(x)}}
  +\norm{\T t{P_{2^N},x}-P_{2^N}(x)}.
\end{displaymath}
Since $P_{2^N}(x)=\sum_{k=0}^N Q_k(x)$,
we get
\begin{multline*}
I\le\norm{\T t{f-P_{2^N},x}-\prn{f(x)-P_{2^N}(x)}}
  +\sum_{k=0}^N\norm{\T t{Q_k,x}-Q_k(x)}\\
=J+\sum_{k=1}^N I_k.
\end{multline*}

Let~$N$ be chosen so that
\begin{equation}\label{eq:delta}
\frac\pi{2^N}<\delta\le\frac\pi{2^{N-1}}.
\end{equation}
We prove that the following inequalities are satisfied
\begin{equation}\label{eq:J}
J\le\Cn M\varphi(\delta)
\end{equation}
and
\begin{equation}\label{eq:Ik}
I_k\le\Cn M\delta^2 2^{2k}\varphi\prn{\frac1{2^k}},
\end{equation}
where constants~$\prevC$ and~$\lastC$ do not depend
on~$f$, $M$, $\delta$ and~$k$.

First we consider~$J$.
Applying Lemma~\ref{th:bound-T} to the function
$\Phi(x)=f(x)-P_{2^N}(x)$, considering that $|t|\le\delta$
we obtain
\begin{multline*}
J\le\norm{\T t{\Phi,x}}+\norm{\Phi(x)}\\
\le\Cn\Big(
  \norm{\Phi}
	+\delta^{2(\gamma_1+\gamma_2)}
	  \normpar{\Phi}{p,\alpha-\gamma_1,\beta-\gamma_2}
	+\delta^{2\gamma_3}
	  \normpar{\Phi}{p,\alpha-\gamma_3,\beta-\gamma_3}\\
+\delta^{2(\gamma_1+\gamma_2+\gamma_3)}
		\normpar{\Phi}{p,\alpha-\gamma_1-\gamma_3,
		  \beta-\gamma_2-\gamma_3}
  \Big),
\end{multline*}
where numbers~$\gamma_1$, $\gamma_2$ and~$\gamma_3$
are chosen by the conditions of Lemma~\ref{th:bound-T}.
Applying Lemma~\ref{lm:rho-sigma}, considering
the conditions of the theorem we obtain
\begin{multline*}
J\le\Cn M\varphi\prn{\frac1{2^N}}\Big(1
	+\delta^{2(\gamma_1+\gamma_2)}2^{-2N(\gamma_1+\gamma_2)}\\
+\delta^{2\gamma_3}2^{-2N\gamma_3}
	+\delta^{2(\gamma_1+\gamma_2+\gamma_3)}
		2^{-2N(\gamma_1+\gamma_2+\gamma_3)}\Big)
\end{multline*}
for $\lambda>\lambda_0+\varepsilon$, where constant~$\lastC$
does not depend on~$f$, $M$ and~$\delta$, and either
$\varepsilon=0$ or $\varepsilon$~is an arbitrary number from
the interval $0<\varepsilon<\frac12$.
Hence this inequality holds true for every
$\lambda>\lambda_0$.
Finally, applying the inequality~\eqref{eq:delta}
and the properties of the function~$\varphi$ we obtain
\begin{displaymath}
J\le\Cn M\varphi\prn{\frac1{2^N}}
\le\Cn M\varphi(\delta).
\end{displaymath}
Thus inequality~\eqref{eq:J} is proved.

Now we prove the inequality~\eqref{eq:Ik}.
It can be proved that~\cite{p-berisha:east-98}
\begin{displaymath}
I_k\le\Cn\delta^2 2^{2k}\norm{Q_k},
\end{displaymath}
where constant~$\lastC$ does not depend
on~$M$, $\delta$ and~$k$.
Hence inequality~\eqref{eq:Qk} yields
\begin{displaymath}
I_k\le\Cn M\delta^2 2^{2k}\varphi\prn{\frac1{2^k}}.
\end{displaymath}
Inequality~\eqref{eq:Ik} is proved.

Inequalities~\eqref{eq:J} and~\eqref{eq:Ik} imply
\begin{displaymath}
I\le\Cn M\brc{\varphi(\delta)
  +\delta^2\sum_{k=1}^N2^{2k}\varphi\prn{\frac1{2^k}}}.
\end{displaymath}
Note that
\begin{displaymath}
\sum_{j=2^k}^{2^{k+1}-1}j\varphi\prn{\frac1j}
\ge C_{\varphi,1}^{-1}C_{\varphi,2}^{-1}
  \varphi\prn{\frac1{2^k}}
	\sum_{j=2^k}^{2^{k+1}-1}j
\ge\Cn2^{2k}\varphi\prn{\frac1{2^k}}.
\end{displaymath}
Hence considering the inequality~\eqref{eq:phi-4}
we have
\begin{multline*}
\sum_{k=1}^N2^{2k}\varphi\prn{\frac1{2^k}}
  \le\Cn\sum_{k=1}^N\sum_{j=2^k}^{2^{k+1}-1}
      j\varphi\prn{\frac1j}
  \le\lastC\sum_{k=1}^{2^{N+1}}k\varphi\prn{\frac1k}\\
\le\Cn2^{2(N+1)}\varphi\prn{\frac1{2^{N+1}}}
  \le\Cn2^{2N}\varphi\prn{\frac1{2^N}}.
\end{multline*}
There from, applying the inequality~\eqref{eq:delta}
we get
\begin{displaymath}
I\le\Cn M\prn{\varphi(\delta)
    +\delta^2 2^{2N}\varphi\prn{\frac1{2^N}}}
\le\Cn M\varphi(\delta).
\end{displaymath}

This way for $0<|t|\le\delta$ we proved that
\begin{displaymath}
\norm{\T t{f,x}-f(x)}\le\lastC\varphi(\delta),
\end{displaymath}
where constant~$\lastC$ does not depend on~$f$ and~$t$.
Taking into consideration that $\T0{f,x}=f(x)$,
we conclude that this inequality also holds for $t=0$.
Thus the last inequality implies
\begin{displaymath}
\w\le\lastC M\varphi(\delta).
\end{displaymath}

Theorem~\ref{th:EsubH} is proved.
\end{proof}

\begin{thm}\label{th:coincidence}
Let be given numbers~$p$, $\alpha$ and~$\beta$ such that
$\allp$;
\begin{alignat*}3
\frac12      &<\alpha\le2,
  &\quad \frac12      &<\beta\le2
    &\quad &\text{for $p=1$},\\
1-\frac1{2p} &<\alpha<3-\frac1p,
  &\quad 1-\frac1{2p} &<\beta<3-\frac1p
	  &\quad &\text{for $1<p<\infty$},\\
1            &\le\alpha<3,
  &\quad 1            &\le\beta<3
	  &\quad &\text{for $p=\infty$}.
\end{alignat*}
Let~$\varphi$ be a function of modulus of continuity type
such hat inequality~\eqref{eq:phi-3} for
\begin{displaymath}
\lambda_0
=\max\brc{
  |\alpha-\beta|,\alpha-\frac32+\frac1{2p},
  \beta-\frac32+\frac1{2p}},
\end{displaymath}
and inequality~\eqref{eq:phi-4} are satisfied.
Let $f\in\Lp$. For
\begin{displaymath}
\E\le\Cn\varphi\prn{\frac1n},
\end{displaymath}
it is necessary and sufficient that
\begin{displaymath}
\w\le\Cn\varphi(\delta),
\end{displaymath}
where constants~$\prevC$ and~$\lastC$ do not depend
on~$f$, $n$ and~$\delta$.
\end{thm}

Theorem~\ref{th:coincidence} is implied directly by
Theorems~\ref{th:HsubE} and~\ref{th:EsubH}.

\begin{thm}\label{th:E-wD}
Let be given numbers~$p$, $\alpha$, $\beta$, $\nu$, $\mu$, $r$,
$\nu_0$ and~$\mu_0$ such that $\allp$, $r\in\mathbb N\cup\{0\}$,
$\nu\ge\mu\ge-\frac12$,
\begin{displaymath}
\nu_0=\min\brc{\nu,\frac52-\frac1{2p}}, \quad
\mu_0=\min\brc{\mu,\frac52-\frac1{2p}};
\end{displaymath}
\begin{enumerate}
\item if $\nu=\mu>\frac12$, then $\alpha=\beta$,
and
\begin{alignat*}2
\frac12      &<\alpha\le\nu_0
  &\quad &\text{for $p=1$},\\
1-\frac1{2p} &<\alpha<\nu_0+\frac12-\frac1{2p}
  &\quad &\text{for $1<p<\infty$},\\
1            &\le\alpha<\nu_0+\frac12
  &\quad &\text{for $p=\infty$};
\end{alignat*}
\item if $\nu>\mu>\frac12$, then $\nu-\mu>\alpha-\beta\ge0$,
and
\begin{alignat*}2
\frac12      &<\beta\le\mu_0
  &\quad &\text{for $p=1$},\\
1-\frac1{2p} &<\beta<\mu_0+\frac12-\frac1{2p}
  &\quad &\text{for $1<p<\infty$},\\
1            &\le\beta<\mu_0+\frac12
  &\quad &\text{for $p=\infty$};
\end{alignat*}
\end{enumerate}
Let~$\varphi$ be a function of modulus of continuity type
such that inequality~\eqref{eq:phi-3} for
\begin{displaymath}
\lambda_0
=\max\brc{
  |\alpha-\beta|,\alpha-\frac32+\frac1{2p},
	\beta-\frac32+\frac1{2p}},
\end{displaymath}
and inequality~\eqref{eq:phi-4} are satisfied.
Let $f\in\Lp$. Necessary and sufficient condition for
\begin{displaymath}
\E\le\frac{\Cn}{n^{2r}}\varphi\prn{\frac1n}
\end{displaymath}
is that the function~$f(x)$ has the derivative of
order $2r$ inside of the interval $(-1,1)$
satisfying the condition
\begin{displaymath}
\wpar{\Dx^r f,\delta}\le\Cn\varphi(\delta),
\end{displaymath}
where constants~$\prevC$ and~$\lastC$ do not depend
on~$f$, $n$ and~$\delta$, while $\Dx^0f(x)=f(x)$.
\end{thm}

Theorem~\ref{th:E-wD} is implied by
Theorems~\ref{th:coincidence} and~\ref{th:E-ED}.

Note that for $\varphi(\delta)=\delta^\lambda$,
$2\lambda_0<\lambda<2$ and $r=0$ Theorem~\ref{th:E-wD}
is proved in~\cite{p-berisha:east-98}.

\bibliographystyle{amsplain}
\bibliography{maths}

\end{document}